\documentclass[12pt,  leqno]{article} 
 
\usepackage{amsmath,amsthm}
\usepackage{amssymb}

\usepackage[tiny]{titlesec}
\titleformat{\section}[runin]{\bfseries}{\thesection.}{1.0ex}{}[]

\usepackage{cite}
\usepackage{url}
\usepackage{hyperref}

\usepackage[T1]{fontenc}
\def\cprime{$'$}

\newtheorem{thm}{Theorem}[section]

\newtheorem{lem}[thm]{Lemma}

\theoremstyle{definition}

\newtheorem{rem}[thm]{Remark}
\newtheorem{exa}[thm]{Example}

\newtheorem*{exa*}{Example}

\numberwithin{equation}{section}

\begin{document}


\baselineskip=17pt


\title{On measurability of Banach indicatrix}

\author{
Nikita Evseev\\
Novosibirsk State University \\
E-mail: nikita@phys.nsu.ru
}

\date{}

\maketitle
	

\renewcommand{\thefootnote}{}



\renewcommand{\thefootnote}{\arabic{footnote}}
\setcounter{footnote}{0}

\begin{abstract}
We prove measurability of the multiplicity function for a measurable mapping
of metric measure spaces.
\end{abstract}


\section{Introduction.}
Given two metric measure spaces $X$, $Y$.
Let $f:X\to Y$ be a measurable mapping and $A\subset X$.
The Banach indicatrix (multiplicity function) is defined as 
$$
N(y,f,A) = \#\{x\in A \mid f(x) = y\},
$$
i.e. the number of elements of $f^{-1}(y)$ in $A$ (possible $\infty$).
In case $A=X$ note $N(y,f,X) = N(y,f)$.
The question under our consideration is following: \textit{is the function $N(y,f,A)$ measurable?}

Let us briefly discuss some results and examples.
The measurability of the multiplicity function for a continuous function $f:[a,b]\to\mathbb R$
was proved by Banach in \cite[Th\'{e}or\`{e}m 1.1]{B}.
Whereas \cite[Th\'{e}or\`{e}m 1.2]{B} states that $\int_a^bN(y,f)\, dy$ is equal
to the total variation $TV(f, [a,b])$.
Together  Th\'{e}or\`{e}ms 1.1 and 1.2 are named the Banach indicatrix theorem
(see \cite[p. 225--227]{N}, \cite[p. 66--72]{L}, \cite[177-178]{BC}).
There are further generalizations of this result, see for example \cite{TS, WS, RL} and the bibliography therein.

The Banach indicatrix play a role in the change of variables formula 
$$
\int_A(u\circ f)|J(x,f)|\, dx = \int_{\mathbb R^n}u(y)N(y,f,A)\, dy.
$$
In \cite{H} the formula was obtained under minimal assumptions: the a.e. existence of
approximative partial derivatives. 
In particular, the measurability of $N(y,f,A)$ was proved.

In \cite[IV.1.2]{RR} the multiplicity function of a continuous transform 
was studied in detail. 
See also \cite[p. 272]{GR} for further investigation. 
The treatment in the setting of metric spaces is given in \cite[2.10.10--15]{F}.

This note aims to show the measurability of the Banach indicatrix for a measurable mapping
(Theorem \ref{theorem:indicatrix_measurability}).
The proof of Lemma \ref{lemma:lemma1} is based upon ideas of the original proof of \cite[Th\'{e}or\`{e}m 1]{B}. While Lemma \ref{lemma:lemma2} is from authors's joint work with Professor S. K. Vodopyanov. 
 
\section{Assumptions and result.}
Let $(X,d_X,\mu_X)$ is a complete, separable metric space with a measure.
Additionally $X$ is supposed to be geometrically doubling: there is a constant $\lambda\in\mathbb N$ such that 
every ball $B(x,r)=\{z\in X \mid d_X(x,z)<r\}$ can be covered by at most $\lambda$
balls $B(x,r/2)$ of half radius.
Measure $\mu_X$ is a Borel regular measure such that each ball has finite measure. 
Assume $(Y,d_Y,\mu_Y)$ is a separable metric measurable space.

The mapping $f:X\to Y$ is a $\mu_X$-measurable if and only if $f$ is defined $\mu_X$-almost everywhere on X
and $f^{-1}(E)$ is $\mu_X$-measurable whenever $E$ is open subset of $Y$ \cite[2.3.2]{F}.

\begin{thm}\label{theorem:indicatrix_measurability}
Let $f:X\to Y$ be a $\mu_X$-measurable mapping, and $A\subset X$ be a Borel set.
Then $f$ can be redefined on a set of $\mu_X$-measure zero in such a way that 
the Banach indicatrix $N(y,f,A)$ is a $\mu_Y$-measurable function.
\end{thm}

\begin{exa}
Let $C\subset\mathbb R$ denotes the Cantor set and $V\subset\mathbb R$ denotes the Vitaly non-measurable set.
There is a bijection $f:C\to V$.
Define the function
$$
\tilde{f}(x) = \begin{cases} f(x), &\text{ if } x\in C,\\
0, &\text{ if } x \not\in C,
\end{cases}
$$
which is measurable. 
But at the same time the multiplicity function $N(y,\tilde{f},A)$
can not be measurable as it coincides with characteristic function of the non-measurable set $V$ 
on $\mathbb R\setminus\{0\}$.
\end{exa}

\noindent
\textit{\text{Dyadic system.}} 
We involve a system of dyadic cubes.
Namely a family 
$$
\{Q^k_{\alpha} \mid k\in\mathbb Z, \alpha\in\mathcal A_k\subset\mathbb N \}
$$ 
of Borel sets with parameters 
$\delta\in(0,1)$, $0<c\leq C<\infty$ and centres $\{x_\alpha^k\}$, meeting the following properties:\\
1) If $l\geq k$ then either $Q^l_\beta\subset Q^k_\alpha$ or $Q^l_\beta\cap Q^k_\alpha = \emptyset$;\\
2) For each $k\in\mathbb Z$ $X = \bigcup\limits_{\alpha\in\mathcal A_k}Q^k_\alpha$ is a disjoint union;\\
3) $B(x^k_\alpha,c\delta^k)\subset Q^k_\alpha\subset B(x^k_\alpha,C\delta^k)$;\\
4) If $l\geq k$ and $Q^l_\beta\subset Q^k_\alpha$ then $B(x^l_\beta,C\delta^l)\subset B(x^k_\alpha,C\delta^k)$.\\
This specific dyadic system in doubling quasi-metric spaces was constructed in \cite{HK}
and generalize the dyadic cubes in the Euclidean space. 

\section{Measurability establishing.} 
Before proceeding with the Theorem~\ref{theorem:indicatrix_measurability}
we need following tow lemmas.  

\begin{lem}\label{lemma:lemma1}
Let $A\subset X$ is a Borel set and $f:X\to Y$ is a $\mu_X$-measurable mapping possessing 
the following property: $f(B)$ is $\mu_Y$-measurable whenever $B\subset A$ is a Borel set.
Then $N(y,f,A)$ is a $\mu_Y$-measurable function.
\end{lem}
\begin{proof}
Take a system $\{Q^k_\alpha\}$ of dyadic cubes on X,
and define a family of functions
$$
L^k_\alpha(y) = \chi_{f(Q^k_\alpha\cap A)}(y).
$$
Functions $L^k_\alpha(y)$ are non-negative and $\mu_Y$-measurable
(as characteristic functions of $\mu_Y$-measurable sets $f(Q^k_\alpha\cap A)$).
Therefore the sum
$$
N_k(y) = \sum_{\alpha\in\mathcal A_k}L^k_\alpha(y)
$$
is also measurable. Thus the sequence of measurable functions $\{N_k(y)\}$
is non-decreasing and the pointwise limit 
$$
N^*(y) = \lim\limits_{k\to\infty}N_k(y)
$$ 
exists and is a $\mu_Y$-measurable function.

Note that $N_k(y)$ counts on how many of the sets $Q^k_\alpha\cap A$ the function $f$
attains the value $y$ at least once.
So for each $k$ $N(y,f,A)\geq N_k(y)$ and $$N(y,f,A)\geq N^*(y).$$

Prove the reverse inequality.
Let $q$ be an integer such that $N(y,f,A)\geq q$.
Then there exist $q$ different points $x_1,\dots,x_q\subset A$ such that $f(x_j) = y$.
If $k$ is large enough so that points $x_1,\dots,x_q$ are in separated cubes $\{Q^k_{\alpha_j}\}$,$j=1,\dots,q$,
then $N_k(y)\geq q$.
This shows $N^*(y)\geq N(y,f,A)$ and
$$
N^*(y) = N(y,f,A).
$$   
\end{proof}

\begin{lem}\label{lemma:lemma2}
Let $f:X\to Y$ be a $\mu_X$-measurable mapping.
Then there is an increasing sequence of closed sets $\{T_k\}\subset X$ such that 
$f$ is continuous on every $T_k$ and $\mu_X\bigl(X\setminus\bigcup\limits_kT_k\bigr)=0$.
\end{lem}
\begin{proof}
Let $\{Q_\alpha\}$ be a collection of dyadic cubes of one generation and 
 $$X = \bigcup\limits_{\alpha=1}^{\infty} Q_\alpha \quad \text{ -- disjoint union.}$$ 

By Luzin's theorem \cite[2.3.5]{F} there is a closed set $C^1_\alpha\subset Q_\alpha$
such that $f$ is continuous on $C^1_\alpha$ and $\mu_X(Q_\alpha\setminus C^1_\alpha)<1$.
Similarly $f$ continuous on $C^2_\alpha\subset Q_\alpha\setminus C^1_\alpha$ and $\mu_X((Q_\alpha\setminus C^1_\alpha)\setminus C^2_\alpha)<1/2$
and so on.
This yields a sequence $\{C^j_\alpha\}$ of closed sets.

Put $$P^k_\alpha = \bigcup\limits_{i=1}^k C^i_\alpha,$$ then $P^k_\alpha\subset P^{k+1}_\alpha$
and the mapping $f$ is continuous on each $P^k_\alpha$. 
Furthermore $\mu_X(Q_\alpha\setminus P^k_\alpha)<1/k$ and hence 
$\mu_X(Q_\alpha\setminus \bigcup\limits_{k} P^k_\alpha)=0$. 

Now defining $$T_j = \bigcup\limits_{\alpha=1}^jP^j_\alpha,$$ we get an increasing sequence of closed sets.
In particular, $\mu_X(Q_\alpha\setminus \bigcup\limits_{j}T_j)=0$ 
since $\bigcup\limits_{j}P^j_\alpha\subset \bigcup\limits_{j}T_j$.
Then 
$$X\setminus \bigcup\limits_{j=1}^\infty T_j = \bigcup\limits_{\alpha=1}^{\infty}(Q_\alpha\setminus \bigcup\limits_{j=1}^{\infty}T_j).$$
Consequently the set $X\setminus \bigcup\limits_{j}T_j$ is of $\mu_X$-measure zero
as a countable union of negligible sets.
\end{proof}

\begin{proof}[Proof of Theorem \ref{theorem:indicatrix_measurability}]
Let $\{T_k\}$ be a sequence of closed sets from Lemma~\ref{lemma:lemma2}.
Observe that an image of each Borel set $B\subset T_k$ is $\mu_Y$-measurable 
since $f$ is continuous on $T_k$ \cite[2.2.13]{F}.
This puts us in a position to apply Lemma~\ref{lemma:lemma1} to deduce that $N(y,f,A\cap T_k)$ is a $\mu_Y$-measurable function.
The sequence $N(y,f,A\cap T_k)$ is non-decreasing and hence
$$N\biggl(y,f,A\cap A\cap \bigcup\limits_k T_k\biggr) = \lim\limits_{k\to\infty}N(y,f,A\cap T_k)$$
is a $\mu_Y$-measurable function.

Take a point $y_0\in Y$ and redefine 
$
f(x) = y_0 \text{ for } x\in X\setminus \bigcup\limits_k T_k.
$  
\end{proof}

\begin{rem}
Note that Theorem \ref{theorem:indicatrix_measurability} requires that the set $A$ be a Borel set.
On the other hand one can prove an analogous assertion for measurable set $A$ however assuming that mapping $f$
satisfies the Luzin $\mathcal N$-property (because in this case the continuous image of every measurable set is measurable and Lemma \ref{lemma:lemma1} is appliable). 
\end{rem} 



\end{document}